\documentclass[11pt,reqno]{amsart}
\usepackage{amsmath, mathtools}
\usepackage{amsthm}
\usepackage{amssymb}
\usepackage{enumerate}
\usepackage{amscd}
\usepackage{color}
\usepackage{pb-diagram}
\usepackage{graphicx}
\usepackage[all, cmtip]{xy}
\usepackage{soul}
\usepackage[colorinlistoftodos]{todonotes}
\usepackage{esint}
\usepackage{hyperref}
\hypersetup{pdftex,colorlinks=true,allcolors=blue}
\usepackage{hypcap}
\usepackage[titletoc]{appendix}
\usepackage{comment}
\usepackage{float}
\restylefloat{table}
\usepackage{tikz}
\usepackage{tikz-cd}
\usetikzlibrary{calc,decorations.pathmorphing,shapes}

\usepackage[colorinlistoftodos]{todonotes}
\usepackage[normalem]{ulem}
\usepackage{soul,xcolor}
\setstcolor{red}
\makeatletter
\def\uwave{\bgroup \markoverwith{\lower3.5\p@\hbox{\sixly \textcolor{red}{\char58}}}\ULon}
\font\sixly=lasy6 
\makeatother

\newenvironment{red}{\relax\color{red}}{\relax}
\newenvironment{blue}{\relax\color{blue}}{\hspace*{.5ex}\relax}

\newcommand{\ber}{\begin{red}}
	\newcommand{\er}{\end{red}}
\newcommand{\beb}{\begin{blue}}
	\newcommand{\eb}{\end{blue}}
\newcounter{sarrow}

\usepackage[colorinlistoftodos]{todonotes}
\usepackage[normalem]{ulem}
\usepackage{xcolor}
\setstcolor{red}

\theoremstyle{plain}
\newtheorem{theorem}{Theorem}[section]
\newtheorem{lemma}[theorem]{Lemma}
\newtheorem{corollary}[theorem]{Corollary}
\newtheorem{proposition}[theorem]{Proposition}

\theoremstyle{definition}
\newtheorem{definition}[theorem]{Definition}
\newtheorem{remark}[theorem]{Remark}
\newtheorem{example}[theorem]{Example}

\numberwithin{equation}{section}

\newcommand{\Ric}{\textup{Ric}}

\newcommand{\twobar}{/\kern-0.5em/}
\newcommand{\threebar}{/\kern-0.5em/\kern-0.5em/}

\DeclareMathOperator{\Rm}{Rm}
\DeclareMathOperator{\Scal}{Scal}

\newcommand{\KN}{\mathbin{\bigcirc\mspace{-15mu}\wedge\mspace{3mu}}}

\title[Rigidity results from Lichnerowicz Laplacian and applications]{Rigidity results with curvature conditions from Lichnerowicz Laplacian and applications}

\author{Gunhee Cho, Nguyen Thac Dung, and Tran Quang Huy}

\begin{document}

\begin{abstract}
	The Bochner technique is a classical tool in global differential geometry for proving vanishing and rigidity results by exploiting curvature conditions. Building on recent extensions of this method to complete non-compact settings by Petersen and Wink, we investigate $L^Q$-harmonic tensors with $Q>1$ governed by the Lichnerowicz Laplacian on complete Riemannian manifolds. Our results generalize Bochner-type theorems to the non-compact realm, revealing new geometric rigidity phenomena not visible in compact cases.
	
	We establish vanishing theorems under integral curvature bounds and weighted Poincaré inequalities, and derive conditions under which harmonic tensors must vanish. In particular, we show that on Ricci-flat or Einstein manifolds, curvature tensors such as $\mathrm{Rm}$ or the Weyl tensor $W$ vanish identically under natural $L^Q$-integrability and positivity assumptions on the curvature operator. These results imply strong rigidity: flatness in the Ricci-flat case and constant sectional curvature in the Einstein case.
	
	We further apply our framework to closed hypersurfaces in space forms and derive vanishing results for intermediate Betti numbers under positivity conditions on the second fundamental form. Finally, we extend our theory to asymptotically locally Euclidean (ALE) spaces, proving that harmonic Weyl tensors and Codazzi tensors must vanish under curvature positivity and decay conditions. Our analysis also links these results to ADM mass rigidity, establishing new obstructions to nontrivial decaying solutions on ALE 4-manifolds.
\end{abstract}
	
	\maketitle
	
\section{Introduction}

The Bochner technique has long played a fundamental role in global Riemannian geometry, yielding vanishing and rigidity theorems through curvature conditions and harmonic analysis. Building on the classical works of Meyer, Gallot-Meyer, and Gallot~\cite{meyer, GM, Gal}, Petersen and Wink recently introduced a new class of curvature positivity conditions adapted to this method~\cite{PW}. In particular, they demonstrated that on a closed Riemannian manifold $(M^n, g)$ of dimension $n \geq 3$, if the curvature operator is $(n - \ell)$-positive for some $1 \leq \ell \leq \lfloor \frac{n}{2} \rfloor$, then the Betti numbers $b_1(M), \dots, b_\ell(M)$ vanish, and all harmonic $\ell$-forms are parallel. Moreover, under $(n - \ell)$-nonnegativity, they established a rigidity result for harmonic $\ell$-forms. Their framework further extends a theorem of Tachibana, proving that a connected closed Einstein manifold with $\lfloor \frac{n-1}{2} \rfloor$-positive curvature must have constant sectional curvature.

More recently, in~\cite{PW20}, Petersen and Wink extended their analysis to complete non-compact weighted manifolds, proving vanishing results for $L^2$-integrable weighted harmonic forms under curvature positivity conditions and weighted Poincaré inequalities.

Motivated by these developments, the present work investigates the behavior of harmonic tensors—specifically those associated with the Lichnerowicz Laplacian—on complete non-compact Riemannian manifolds. Our goal is to generalize Bochner-type vanishing theorems for $L^Q$ harmonic tensors, moving beyond compactness and leveraging integral curvature bounds and analytic inequalities.

Extending such results to the non-compact setting is not merely a technical generalization but reveals new geometric phenomena that are invisible in the compact case. In particular, non-compact manifolds often exhibit richer asymptotic structures. One especially significant application of our theory arises in the context of asymptotically locally Euclidean (ALE) manifolds. 

Let $(M^n,g)$ be a complete non-compact Riemannian manifold and consider the Lichnerowicz Laplacian defined by $\Delta_L := \nabla^* \nabla + c \mathfrak{R i c}$ for $c > 0$. Our first main result establishes vanishing of harmonic tensors under integral decay assumptions and curvature positivity:

\begin{theorem}\label{thm1}
	Let $(M,g)$ be a complete non-compact Riemannian $n$-manifold. Suppose the curvature tensor is $\lceil \frac{n}{2} \rceil$-nonnegative. Then any harmonic tensor $T$ with respect to the Lichnerowicz Laplacian $\Delta_L$ must vanish if $|T| \in L^Q(M)$ for some $Q > 1$. Here $\hat{T}$ is the associated curvature-type tensor defined in Definition~\ref{def22}.
\end{theorem}

If we assume $g(\Re(\hat{T}), \hat{T}) \geqslant -\kappa|T|^2$ for some $\kappa \geqslant 0$, we can prove a similar result. However, we need to make additional assumptions, specifically that the weighted Poincar'e inequality holds (\cite[Def 0.1]{LPWJ}). This means that there is a non-negative, continuous, and not identically zero weight function $\rho$ on $M$, such that
\begin{align}\label{e1.2}
	\int_M \rho(x) \phi^2(x){\rm d}v \leq \int_M|\nabla \phi|^2{\rm d}v
\end{align}
is valid for any compactly supported smooth function $\phi \in C_0^{\infty}(M)$. It is well known that if \eqref{e1.2} holds true then the volume of $M$ is infinite. All assumptions on a weight function $\rho$ can be regarded as the generalization of the positivity condition of the first Dirichlet eigenvalue $\lambda_1(M)$ \cite{LPWJ}. 

\begin{theorem}\label{thm2}
	Let $(M,g)$ be a connected, complete, non-compact Riemannian $n$-manifold. Suppose $M$ satisfies a weighted Poincaré inequality with weight function $\rho$, and assume for all $(0,k)$-tensors $T$ that
	\[
	g(\mathfrak{R}(\hat{T}), \hat{T}) \geq -\kappa \rho |T|^2, \quad \text{for some } \kappa \geq 0.
	\]
	If $|T| \in L^Q(M)$ for some $Q>1$, and
	\[
	0 \leq \kappa < \frac{4(Q-1)}{c} \min\left\{ \frac{1}{Q^2}, \frac{1}{4 + (Q-2)^2} \right\},
	\]
	then $T$ vanishes identically.
\end{theorem}

These general vanishing results yield powerful rigidity theorems in specific geometric settings. For instance, on Ricci-flat manifolds, the Riemann curvature tensor is harmonic with respect to $\Delta_L$ with $c = \frac{1}{2}$. This yields the following:

\begin{theorem}\label{main4}
	Let $(M^n,g)$ be a complete, connected, non-compact Ricci-flat manifold with $n \geq 3$. Suppose the curvature operator is $\lfloor \frac{n-1}{2} \rfloor$-nonnegative and $|\mathrm{Rm}| \in L^Q(M)$ for some $Q > 1$. Then $\mathrm{Rm} \equiv 0$, and hence $(M,g)$ is flat.
\end{theorem}

Similarly, we obtain rigidity for Einstein manifolds via the behavior of the Weyl tensor and refined Kato-type inequalities:

\begin{theorem}\label{thm:constant-curvature}
	Let $(M^n,g)$ be a complete, connected, non-compact Einstein manifold with $n \geq 4$. Suppose $M$ satisfies a weighted Poincar\'{e} inequality with weight function $\rho$, and the curvature operator satisfies
	\[
	\frac{\mu_1 + \cdots + \mu_{\lfloor \frac{n-1}{2} \rfloor}}{\lfloor \frac{n-1}{2} \rfloor} \geq -\kappa \rho,
	\]
	where $\mu_i$ are the eigenvalues of the curvature operator and $\kappa \geq 0$. If the Weyl tensor $W$ satisfies $|W| \in L^Q(M)$ for some $Q > 1$ and
	\[
	\kappa < \frac{2\left(Q - 1 + \frac{2}{n - 1}\right)}{n - 1} \min\left\{ \frac{1}{Q^2}, \frac{1}{4 + (Q - 2)^2} \right\},
	\]
	then $(M,g)$ has constant sectional curvature.
\end{theorem}

When $\kappa = 0$, this theorem recovers a case of nonnegative curvature operator without requiring the Poincar\'{e} inequality—an observation consistent with results obtained by different techniques in~\cite{colombo}.

We also provide geometric applications to submanifold theory. For an immersed hypersurface $M$ in a space form of constant sectional curvature $K$, we define an $m$-positivity condition on the second fundamental form and derive Betti number vanishing:

\begin{theorem}\label{main5}
	Let $M^n$ be a closed immersed hypersurface in a space form with constant curvature $K$ and $1 \leq p \leq \lfloor \frac{n}{2} \rfloor$. If the symmetric functions $\mu_1 + \cdots + \mu_{n-p} > -(n-p)K$, then $b_1(M) = \cdots = b_p(M) = b_{n-p}(M) = \cdots = b_{n-1}(M) = 0$.
\end{theorem}

\begin{corollary}
	Let $M^n$ be a closed immersed hypersurface as above. If
	\[
	\mu_1 + \cdots + \mu_{n - \lceil \frac{n}{2} \rceil} > -\left(n - \lceil \tfrac{n}{2} \rceil \right)K,
	\]
	then $b_p(M) = 0$ for all $0 < p < n$.
\end{corollary}

We further explore rigidity phenomena on asymptotically locally Euclidean (ALE) manifolds, which arise as canonical noncompact models in differential geometry and mathematical physics, especially in the study of Ricci-flat metrics and gravitational instantons. The Bochner-type techniques developed in this paper extend naturally to the ALE setting, leading to new vanishing and rigidity results for harmonic Weyl tensors and Codazzi tensors under suitable decay. In particular, we show that Ricci-flat ALE 4-manifolds with $L^Q$-integrable self-dual or anti-self-dual Weyl tensor and 1-nonnegative curvature operator must be flat. Moreover, we prove that if a symmetric 2-tensor satisfies the Lichnerowicz equation with divergence-free, trace-free, and fast-decaying conditions, then either it vanishes identically or the ADM mass must be zero (Proposition~\ref{prop:lichnerowicz-adm}). 

\paragraph{Organization of the paper.} 
In Section~\ref{sec0}, we recall analytic preliminaries and tensor estimates relevant to the Lichnerowicz Laplacian. Theorems~\ref{thm1} and~\ref{thm2} are proved in Section~\ref{sec1}. Section~\ref{app} discusses applications to curvature tensors and proves Theorems~\ref{main4} and~\ref{thm:constant-curvature} on Ricci-flat and Einstein manifolds. In Section~\ref{sec5}, we extend our techniques to the setting of submanifolds and prove Theorem~\ref{main5} and related corollaries. Finally, in Section~\ref{sec6}, we investigate applications to asymptotically locally Euclidean (ALE) manifolds, establishing new rigidity results for harmonic Weyl tensors and Codazzi tensors, as well as decay obstructions arising from the positivity of ADM mass.
   
\section{Preliminaries}\label{sec0}

In this section, we recall basic tools from Riemannian geometry necessary for studying curvature operators acting on tensor fields. We focus on how the curvature tensor induces differential operators such as the Lichnerowicz Laplacian, which plays a central role in Bochner-type formulas and rigidity results.

Let \( (M, g) \) be a \( n \)-dimensional Riemannian manifold. The \emph{Riemann curvature tensor} is a $(1,3)$-type tensor defined by
\[
R(X,Y)Z := \nabla_Y \nabla_X Z - \nabla_X \nabla_Y Z + \nabla_{[X,Y]} Z,
\]
where \( \nabla \) is the Levi-Civita connection of \( g \). Let \( \mathcal{T}^{(0,k)}(M) \) denote the bundle of covariant \( (0,k) \)-tensors on \( M \), that is, smooth sections of the \( k \)-fold tensor product of the cotangent bundle \( T^*M \). Here, for fixed vector fields $X, Y$, the curvature $R(X, Y)$ is naturally a $(1,1)$ tensor, and it can also be applied to tensors. For any tensor $T$ we have

$$
\begin{aligned}
	R(X, Y) T & =\nabla_{X}\left(\nabla_{Y} T\right)-\nabla_{Y}\left(\nabla_{X} T\right)+\nabla_{[X, Y]} T,
\end{aligned}
$$

and

$$
\begin{aligned}
	(R(X, Y) T)\left(X_{1}, \ldots, X_{k}\right)= & R(X, Y)\left(T\left(X_{1}, \ldots, X_{k}\right)\right) -T\left(R(X, Y) X_{1}, \ldots, X_{k}\right)  \\
	&- \cdots  -T\left(X_{1}, \ldots, R(X, Y) X_{k}\right).
\end{aligned}
$$

\begin{definition}
	We define the \emph{Weitzenböck curvature operator} on a tensor \( T \in \mathcal{T}^{(0,k)}(M) \) as
	\[
	\mathfrak{R i c}(T)(X_1, \dots, X_k) := \sum_{i=1}^k \sum_{j=1}^n \left( R(X_i, e_j)T \right)(X_1, \dots, e_j, \dots, X_k),
	\]
	where \( \{e_j\} \) is a local orthonormal frame. 
\end{definition}

\begin{definition}
	For any constant \( c > 0 \), we define the \emph{Lichnerowicz Laplacian} acting on tensor fields by
	\[
	\Delta_L := \nabla^* \nabla + c \mathfrak{R i c}.
	\]
	Here, \( \nabla^* \) is the formal adjoint of \( \nabla \), defined by
	\[
	(\nabla^* T)(X_2, \dots, X_k) := - (\nabla_{E_i} T)(E_i, X_2, \dots, X_k),
	\]
	where the summation is taken over a local orthonormal frame \( \{E_i\} \). A tensor field \( T \in \mathcal{T}^{(0,k)}(M) \) is said to be \emph{harmonic} if \( \Delta_L T = 0 \).
\end{definition}

\begin{example}
	The Lichnerowicz Laplacian includes several important special cases:
	\begin{itemize}
		\item[(a)] When \( c = 1 \), \( \Delta_L \) coincides with the Hodge Laplacian on differential forms.
		\item[(b)] When the curvature tensor is divergence free, then the curvature tensor is harmonic with respect to the Lichnerowicz Laplacian with $c=1/2$. i.e., 	
		$$
		\begin{aligned}
			0=\Delta_{L} R=\nabla^{*} \nabla R+\frac{1}{2} \mathfrak{R i c}(R)
		\end{aligned}
		$$
	\end{itemize}
\end{example}

\subsection*{Tensorial Lie Action}

	\begin{definition}\label{def22}

Following Petersen~\cite{Petersen}, one can associate to any tensor \( T \in \mathcal{T}^{(0,k)}(M) \) a natural \((2,0,k)\)-tensor
\[
\hat{T} \in \Lambda^2(M) \otimes \mathcal{T}^{(0,k)}(M),
\]
which captures the infinitesimal action of the orthogonal Lie algebra \( \mathfrak{so}(M) \cong \Lambda^2(M) \) on \( T \). This is defined via the relation
\[
g(L, \hat{T}(X_1, \dots, X_k)) = (L T)(X_1, \dots, X_k)
\quad\text{for all } L \in \mathfrak{so}(M),
\]
where the Lie algebra action is given by
\[
(L T)(X_1, \dots, X_k) = -\sum_{i=1}^k T(X_1, \dots, L X_i, \dots, X_k).
\]
\end{definition}



\subsection*{Curvature Operator on 2-Forms}

The curvature operator
\[
\mathfrak{R} : \Lambda^2(T M) \to \Lambda^2(T M)
\]
is the symmetric endomorphism defined pointwise by
\[
g(\mathfrak{R}(X \wedge Y), Z \wedge W) := \Rm(X, Y, Z, W).
\]
where $ \Rm$ is a $(4,0)$ Riemann curvature tensor. This operator is self-adjoint with respect to the natural inner product on \( \Lambda^2(T_x M) \) induced by the metric \( g \), and reflects the action of the curvature tensor on 2-forms.

Given the curvature endomorphism \( R(X,Y) : T_x M \to T_x M \), which is skew-symmetric for each \( X, Y \in T_x M \), it admits the following decomposition in terms of \( \mathfrak{R} \) and an orthonormal basis \( \{ \Xi_\alpha \} \subset \mathfrak{so}(T_x M) \cong \Lambda^2(T_x M) \):
\[
\begin{aligned}
	R(X,Y) = g(R(X,Y), \Xi_\alpha) \Xi_\alpha = g(\mathfrak{R}(X \wedge Y), \Xi_\alpha) \Xi_\alpha = g(\mathfrak{R}(\Xi_\alpha), X \wedge Y) \Xi_\alpha = -g(\mathfrak{R}(\Xi_\alpha) X, Y) \Xi_\alpha,
\end{aligned}
\]
where \( \mathfrak{R}(\Xi_\alpha) \) acts as a derivation on tensors via the representation \( \Xi_\alpha \in \mathfrak{so}(T_x M) \).

At each point \( x \in M \), we may diagonalize \( \mathfrak{R} \) by choosing an orthonormal basis \( \{ \Xi_\alpha \} \subset \Lambda^2(T_x M) \cong \mathfrak{so}(T_x M) \) consisting of eigenvectors of \( \mathfrak{R} \), with corresponding eigenvalues \( \lambda_\alpha \in \mathbb{R} \). Then, for any \( T_x \in \Lambda^2(T_x M) \), we have
\[
g(\mathfrak{R}(T_x), T_x) = \sum_\alpha \lambda_\alpha \| \Xi_\alpha(T_x) \|^2,
\]
which expresses how curvature distributes energy across 2-form directions.

The curvature operator induces an action on general \((0,k)\)-tensors \( T \in \mathcal{T}^{(0,k)}(M) \) via
\[
\mathfrak{Ric}(T) := - \sum_\alpha \mathfrak{R}(\Xi_\alpha)(\Xi_\alpha T),
\]
where \( \Xi_\alpha T \) denotes the Lie algebra action:
\[
(\Xi_\alpha T)(X_1,\dots,X_k) = -\sum_{i=1}^k T(X_1,\dots,\Xi_\alpha X_i,\dots,X_k).
\]
This gives rise to the following identity, due to Petersen~\cite{Petersen}:

\begin{proposition}\label{ric}
	Let \( S, T \in \mathcal{T}^{(0,k)}(M) \). Then
	\[
	g(\mathfrak{Ric}(S), T) = g(\mathfrak{R}(\hat{S}), \hat{T}),
	\]
	where \( \hat{S}, \hat{T} \in \Lambda^2(M) \otimes \mathcal{T}^{(0,k)}(M) \). In particular, \( \mathfrak{Ric} \) is a self-adjoint operator.
	
	
\end{proposition}

\subsection*{Bochner Formula}

We recall the classical Bochner formula for a tensor field \( T \in \mathcal{T}^{(0,k)}(M) \):
\begin{equation} \label{eq:bochner-classic}
	\Delta \left( \tfrac{1}{2} |T|^2 \right) = |\nabla T|^2 - g( \nabla^* \nabla T, T ).
\end{equation}

The formula above follows from differentiating the pointwise identity \( |T|^2 = g(T, T) \), applying the metric compatibility of \( \nabla \), and take the divergence. 

Now, suppose that \( T \) is harmonic with respect to the Lichnerowicz Laplacian:
\[
\Delta_L T = \nabla^* \nabla T + c \, \mathfrak{Ric}(T) = 0.
\]
Then,
\[
\nabla^* \nabla T = -c \, \mathfrak{Ric}(T).
\]
Substituting into \eqref{eq:bochner-classic}, we obtain the refined Bochner identity:
\begin{equation} \label{eq:bochner-curvature}
	\Delta \left( \tfrac{1}{2} |T|^2 \right) = |\nabla T|^2 + c \cdot g( \mathfrak{Ric}(T), T ).
\end{equation}

Using Proposition~\ref{ric}, the curvature term can be further rewritten via the curvature operator on \( \Lambda^2 \otimes \mathcal{T}^{(0,k)}(M) \) as:
\[
g( \mathfrak{Ric}(T), T ) = g( \mathfrak{R}(\hat{T}), \hat{T} ),
\]
yielding the spectral representation:
\[
\Delta \left( \tfrac{1}{2} |T|^2 \right) = |\nabla T|^2 + c \cdot g( \mathfrak{R}(\hat{T}), \hat{T} ).
\]

\subsection*{Algebraic Curvature tensor}
Suppose that $(V, g)$ is an $n$-dimensional Euclidean vector space and $\mathcal{T}^{(0,k)}(V)$ stands for the space of $(0, k)$-tensors.  We denote the vector space of symmetric $(0,2)$-tensor by ${\rm Sym}^2(V)$. It is well known that there is an orthogonal decomposition
$${\rm Sym}^2(\Lambda^2V)={\rm Sym}^2_B(\Lambda^2V)\oplus\Lambda^4V,$$
where the vector space ${\rm Sym}^2_B(\Lambda^2V)$ consists of all tensors $R\in{\rm Sym}^2(\Lambda^2V)$ satisfying the first Bianchi identity. An $R \in {\rm Sym}^2_B(\Lambda^2V)$ is called an algebraic curvature tensor. The associated algebraic curvature $(0,4)$-tensor $\Rm$ is defined by 
$$ \Rm(x,y,z,w)= R(x \wedge y, z \wedge w), \forall x,y,z,w \in V.$$

For $S,T \in {\rm Sym}^2 (V)$, the Kulkarni-Nomizu product of $S,T$ is given by 
$$(S \KN T)(x,y,z,w)= S(x,z)T(y,w) - S(x,w)T(y,z)+ S(y,w)T(x,z)-S(y,z)T(x,w).$$

Recall that every algebraic $(0, 4)$-curvature tensor $\Rm$ satisfies the orthogonal decomposition
$$ \Rm = \frac{\Scal}{2(n-1)n} g \KN g + \frac{1}{n-2} g  \stackrel{\circ}{\Ric} +W, $$
where $\stackrel{\circ}{\Ric} = \Ric - \frac{\Scal}{n} g$ is traceless Ricci tensor and $W$ denotes the Weyl part. 
    \subsection*{Work of Petersen-Wink}
  The following lemmas due to Petersen and Wink in \cite{PW21} provide a framework to control the curvature term of the Lichnerowicz Laplacian on tensors.
  \begin{lemma}\label{lem2}
  	Let $\Re: \Lambda^2 V \rightarrow \Lambda^2 V$ be an algebraic curvature operator with eigenvalues $\mu_{1} \leq \ldots \leq \mu_{\binom{n}{2}}$ and let $T \in \mathcal{T}^{(0,k)}(V)$. Suppose there is $C \geq 1$ such that
  	$$
  	|L T|^{2} \leq \frac{1}{C}|\hat{T}|^{2}|L|^{2}
  	$$
  	for all $L \in \mathfrak{s o}(V)$. Let $\kappa \leq 0$. 
  	\begin{enumerate}
  		\item If $\frac{1}{\lfloor C\rfloor}\left(\mu_{1}+\ldots+\mu_{\lfloor C\rfloor}\right) \geq \kappa$, then $g(\Re(\hat{T}), \hat{T}) \geq \kappa|\hat{T}|^{2}$.
  		\item If $\mu_{1}+\ldots+\mu_{\lfloor C\rfloor}>0$, then $g(\Re(\hat{T}), \hat{T})>0$ unless $\hat{T}=0$.
  	\end{enumerate}
  \end{lemma}
  \begin{proof} A proof of this lemma is given in \cite[Lemma~2.1]{PW21} (see also Remark 3.9 in \cite{colombo}). However, for the reader's convenience, we include here a proof which is a modification of the proof in \cite{PW21}. As in \cite{PW21}, we have $|\hat{T}|^2=\sum\limits_{\alpha=1}^{\binom{n}{2}}|\Xi_\alpha T|^2$, where $\Xi_\alpha$ is an orthonormal eigenbasis of $\Re$ with respect to eigenvalues $\{\mu_\alpha\}$. Therefore, 
  	$$\begin{aligned}
  		g(\Re(\hat{T}), \hat{T})
  		&=\sum\limits_{\alpha=1}^{\lfloor C\rfloor}\mu_\alpha|\Xi_\alpha T|^2+\sum\limits_{\alpha=\lfloor C\rfloor+1}^{\binom{n}{2}}\mu_\alpha|\Xi_\alpha T|^2\\
  		&\geq\sum\limits_{\alpha=1}^{\lfloor C\rfloor}\mu_\alpha|\Xi_\alpha T|^2+\mu_{\lfloor C\rfloor+1}\sum\limits_{\alpha=\lfloor C\rfloor+1}^{\binom{n}{2}}|\Xi_\alpha T|^2\\
  		&=\sum\limits_{\alpha=1}^{\lfloor C\rfloor}(\mu_\alpha-\mu_{\lfloor C\rfloor+1})|\Xi_\alpha T|^2+\mu_{\lfloor C\rfloor+1}|\hat{T}|^2,
  	\end{aligned}$$
  	where we used the fact that $\{\mu_\alpha\}$ is increasing. Since $\mu_\alpha-\mu_{\lfloor C\rfloor+1}\leq0$ for all $1\leq\alpha\leq \lfloor C\rfloor $, the assumption $|L T|^{2} \leq \frac{1}{C}|\hat{T}|^{2}|L|^{2}
  	$ for all $L \in \mathfrak{s o}(V)$ implies
  	$$\begin{aligned}
  		g(\Re(\hat{T}), \hat{T})
  		&\geq\sum\limits_{\alpha=1}^{\lfloor C\rfloor}(\mu_\alpha-\mu_{\lfloor C\rfloor+1})|\frac{\hat{T}}{C}+\mu_{\lfloor C\rfloor+1}|\hat{T}|^2\\
  		&=\mu_{\lfloor C\rfloor+1}\left(1-\frac{\lfloor C\rfloor}{C}\right)\hat{T}+\frac{\sum\limits_{\alpha=1}^{\lfloor C\rfloor}\mu_\alpha}{\lfloor C\rfloor}\left(\frac{\lfloor C\rfloor}{C}\hat{T}\right)\\
  		&\geq \kappa\left(1-\frac{\lfloor C\rfloor}{C}\right)\hat{T}+\kappa\frac{\lfloor C\rfloor}{C}\hat{T}=\kappa\hat{T},
  	\end{aligned}$$
  	where we used $\mu_{\lfloor C\rfloor+1}\geq\frac{\mu_{1}+\ldots+\mu_{\lfloor C\rfloor}}{\lfloor C\rfloor}\geq\kappa$ in the last inequality. We are done.
  \end{proof}
  
  The next lemma allows us to estimate $|LT|^2$ for various types of tensors. 
  
  \begin{lemma}\label{lem1}
  	Let $(V,g)$ be an $n$-dimensional Euclidean vector space and $L \in \mathfrak{s o}(V)$. The followings hold: 
  	\begin{itemize}
  		\item[(a)] Every $T \in T^{(0,k)}(V)$ satisfies 
  		$$
  		|L T|^{2} \leq k^{2}|T|^{2}|L|^{2}.
  		$$
  		\item[(b)] Every $\ell$-form $\omega$ 
  		$$
  		|L \omega|^{2} \leq \min \{\ell, n-\ell\}|\omega|^{2}|L|^{2}.
  		$$
  		\item[(c)]Every algebraic curvature $(0,4)$-tensor $\Rm$ satisfies 
  		$$
  		|L \mathrm{Rm}|^{2} \leq 8|\stackrel{\circ}{\Rm}|^2|L|^{2},
  		$$
  		where $\stackrel{\circ}{\Rm}$ is traceless algebraic curvature $(0,4)$-tensor given by $$\stackrel{\circ}{\Rm} = \Rm - \frac{\Scal}{2(n-1)n} g \KN g.$$
  	\end{itemize}
  \end{lemma}
  \begin{proof}
  	See \cite[Lemma~2.2]{PW21}.
  \end{proof}
  Finally, we need the following result about the relation between $|\hat{T}|$ and $|T|$ for several types of tensors above. 
  
  \begin{proposition}\label{prop1}
  	Let $(V,g)$ be an $n$-dimensional Euclidean vector space. The followings hold: 
  	\begin{itemize}
  		\item[(a)] Every $\ell$-form $\omega$ satisfies 
  		$$ |\hat{\omega}|^2=\ell(n-\ell) |\omega|^2.$$ 
  		\item[(b)] Every algebraic $(0,4)$-curvature tensor $\Rm$ and every $R\in{\rm Sym}^2_B(\Lambda^2V)$ satisfies 
  		$$|\widehat{\Rm}|^2 =|\widehat{\stackrel{\circ}{\Rm}}|^2=4(n-1)|\stackrel{\circ}{\Rm}|^2-8|\stackrel{\circ}{\Ric}|^2.$$
  		
  		In particular $\widehat{\Rm}=0$ if and only if $\Rm= \frac{\kappa}{2} g \KN g$ for some $\kappa \in \mathbb{R}$. 
  	\end{itemize}
  \end{proposition}
  \begin{proof}
  	See \cite[Proposition~2.5]{PW21}.
  \end{proof} 

	\section{Rigidity theorems for harmonic tensors}\label{sec1}
To begin with let us introduce the following computational lemma.
\begin{lemma}\label{Q}
	Let $1<Q\leq 2$, and $\eta, \epsilon>0$. For any smooth non-negative functions $u, \varphi$, we have 
	$$
	\left|\nabla\left(\varphi u(u+\delta)^{\frac{Q-2}{2}}\right)\right|^2 \leq\left[\left(\frac{Q-2}{2}\right)^2+1+\varepsilon\right] \varphi^2(u+\delta)^{Q-2}|\nabla u |^2+\left(1+\frac{1}{\varepsilon}\right)|\nabla \varphi|^2u^Q .
	$$
\end{lemma}
\begin{proof}In fact, this result was proved in \cite{DDH}. Since this inequality is used at many places in this paper, we include a detail of proof for inconvenience. By a direct calculation, it is easy to see that 
	\begin{align}\label{k3.15}
		& \left|\nabla\left(\varphi u(u+\delta)^{\frac{Q-2}{2}}\right)\right|^2 \nonumber \\
		= & \left(\frac{Q-2}{2}\right)^2 \varphi^2(u+\delta)^{Q-4}u^2|\nabla u |^2+(u+\delta)^{Q-2}|\nabla \varphi|^2u^2+(u+\delta)^{Q-2} \varphi^2|\nabla u|^2 \nonumber\\
		& +(Q-2)(u+\delta)^{Q-3} \varphi^2|\nabla u|^2 \nonumber\\
		& +(Q-2)u^2(u+\delta)^{Q-3} \varphi\langle\nabla u, \nabla \varphi\rangle+2 \varphi u(u+\delta)^{Q-2}\langle\nabla u, \nabla \varphi\rangle \nonumber\\
		\leq & \left[\left(\frac{Q-2}{2}\right)^2 +1\right](u+\delta)^{Q-2} \varphi^2|\nabla u|^2+(u+\delta)^{Q-2}u^2|\nabla \varphi|^2 
		\nonumber\\
		&+(Qu+2 \delta)\varphi u(u+\delta)^{Q-3}|\nabla u||\nabla \varphi|\nonumber\\
		\leq & \left[\left(\frac{Q-2}{2}\right)^2 +1\right](u+\delta)^{Q-2} \varphi^2|\nabla u|^2+u^{Q}|\nabla \varphi|^2 
		\nonumber\\
		&+2(u+\delta)\varphi u(u+\delta)^{Q-3}|\nabla u||\nabla \varphi|
	\end{align}
	Here we used $(u+\delta)^{Q-4}u^2\leq(u+\delta)^{Q-2}, Q-2\leq0$ and the Cauchy-Schwarz inequality in the second inequality, and $(u+\delta)^{Q-2}u^2\leq u^Q, Q\leq 2$ in the last inequality. Now we estimate the last term of the above inequality. Note that for any $\varepsilon>0$, using $2xy\leq\varepsilon x^2+\frac{1}{\varepsilon}y^2, x,y\in\mathbb{R}$, we have
	$$
	\begin{aligned}
		2 \varphi|\omega|(|\omega|+\delta)^{Q-2}|\nabla| \omega| ||\nabla \varphi| & \leq \varepsilon \varphi^2(|\omega|+\delta)^{Q-2}|\nabla| \omega| |^2+\frac{1}{\varepsilon}|\nabla \varphi|^2|\omega|^2(|\omega|+\delta)^{Q-2} \\
		& \leq \varepsilon \varphi^2\left(|\omega|+\delta\right)^{Q-2}|\nabla| \omega \|^2+\frac{1}{\varepsilon}|\nabla \varphi|^2|\omega|^Q,
	\end{aligned}
	$$
	where we used $Q-2<0$ and $|\omega|<|\omega|+\delta$  in the last inequality. Plugging this inequality into \eqref{k3.15}, we are done.
\end{proof}
Using this lemma, we are ready to prove our theorems.
\begin{proof}[\textsc{Proof of Theorem \ref{thm1}}]
	Let $T$ be a harmonic tensor. The Bochner formula for harmonic tensors implies
	\[
	\Delta \left( \frac{1}{2}|T|^{2} \right) = |\nabla T|^{2} + c \cdot g(\Re(\hat{T}), \hat{T}).
	\]
	Since $g(\Re(\hat{T}), \hat{T}) \geqslant 0$, and applying the Kato inequality \( |\nabla T|^2 \geq |\nabla |T||^2 \), we obtain:
	\begin{equation} \label{eq1}
		\frac{1}{2} \Delta |T|^2 \geq |\nabla |T||^2.
	\end{equation}
	
	Let $q \in \mathbb{R}^+$ and let \( \varphi \in C_c^\infty(M) \) be a smooth nonnegative function with compact support. Multiply both sides of \eqref{eq1} by \( \varphi^2 |T|^q \) and integrate over \( M \):
	\[
	\frac{1}{2} \int_M \varphi^2 |T|^q \Delta |T|^2 \geq \int_M \varphi^2 |T|^q |\nabla |T||^2.
	\]
	
	Apply integration by parts on the left-hand side:
	\begin{align*}
		\frac{1}{2} \int_M \varphi^2 |T|^q \Delta |T|^2 
		&= -\frac{1}{2} \int_M \left\langle \nabla(\varphi^2 |T|^q), \nabla |T|^2 \right\rangle \\
		&= -\int_M \varphi |T|^q \left\langle \nabla \varphi, \nabla |T|^2 \right\rangle 
		- \frac{q}{2} \int_M \varphi^2 |T|^{q-1} \left\langle \nabla |T|, \nabla |T|^2 \right\rangle.
	\end{align*}
	
	Since \( \nabla |T|^2 = 2 |T| \nabla |T| \), we have:
	\begin{align*}
		\frac{1}{2} \int_M \varphi^2 |T|^q \Delta |T|^2
		= -2 \int_M \varphi |T|^{q+1} \left\langle \nabla \varphi, \nabla |T| \right\rangle
		- q \int_M \varphi^2 |T|^q |\nabla |T||^2.
	\end{align*}
	
	Combining this with the inequality above, we obtain:
	\[
	-2 \int_M \varphi |T|^{q+1} \left\langle \nabla \varphi, \nabla |T| \right\rangle
	- q \int_M \varphi^2 |T|^q |\nabla |T||^2 
	\geq \int_M \varphi^2 |T|^q |\nabla |T||^2.
	\]
	
	Rearranging terms gives:
	\begin{equation} \label{eq2}
		(q+1) \int_M \varphi^2 |T|^q |\nabla |T||^2 
		\leq -2 \int_M \varphi |T|^{q+1} \left\langle \nabla \varphi, \nabla |T| \right\rangle.
	\end{equation}
	
	Now apply the Cauchy–Schwarz and Young's inequality:
	\[
	2 \left| \varphi |T|^{q+1} \left\langle \nabla \varphi, \nabla |T| \right\rangle \right|
	\leq \varepsilon \varphi^2 |T|^q |\nabla |T||^2 + \frac{1}{\varepsilon} |\nabla \varphi|^2 |T|^{q+2}, \quad \text{for any } \varepsilon > 0.
	\]
	
	Substituting this into \eqref{eq2} (in absolute value) gives:
	\[
	(q+1) \int_M \varphi^2 |T|^q |\nabla |T||^2 
	\leq \varepsilon \int_M \varphi^2 |T|^q |\nabla |T||^2 + \frac{1}{\varepsilon} \int_M |\nabla \varphi|^2 |T|^{q+2}.
	\]
	
	Rearranging:
	\[
	(q+1 - \varepsilon) \int_M \varphi^2 |T|^q |\nabla |T||^2 
	\leq \frac{1}{\varepsilon} \int_M |\nabla \varphi|^2 |T|^{q+2}.
	\]
	
	Choosing \( \varepsilon \in (0, q+1) \), we obtain:
	\begin{equation} \label{eq3}
		\int_M \varphi^2 |T|^q |\nabla |T||^2 
		\leq C \int_M |\nabla \varphi|^2 |T|^{q+2},
	\end{equation}
	where \( C = \frac{1}{q+1 - \varepsilon} \cdot \frac{1}{\varepsilon} \) is a constant depending only on \( q \) and \( \varepsilon \).
	
	Now, for $Q \geq2$, we can choose $q\geq0$ such that $Q=q+2$ and also choose $\varphi$ satisfying
	$$
	\varphi= \begin{cases}1 & \text { on } B(R) \\ 0 & \text { on } M \backslash B(2 R)\end{cases}
	$$
	and $|\nabla \varphi| \leq \frac{2}{R}$. The inequality \eqref{eq3} implies
	$$\int_{M} \varphi^2 |T|^{Q-2} |\nabla |T||^2 \leqslant \frac{4C}{R^2} \int_{M} |T|^{Q}.$$
	Let $R \to \infty$ and since $|T| \in L^Q(M)$, then $|T|$ is constant on each connected component of $M$. Now, since the Ricci curvature is bounded from below by the sum of the lowest $(n-1)$ eigenvalues of the curvature operator, the curvature assumption implies ${\rm Ric}\geq 0$ (see also Remark 1.10 in \cite{PW20}). Hence, due to non-compactness of $M$, a lower bound estimate in \cite{Yau76} yields the volume of the geodesic ball $B_o(R)$ has at least linear growth. Consequently, $M$ is of infinite volume. Since $|T|$ is constant and $|T|\in L^Q(M)$, we conclude that $T$ must be vanishing. 
	
	When $1<Q< 2$, it is easy to see that \eqref{eq1} implies $|T|\Delta|T|\geq0$. For any $\delta>0$, multiplying both side of \eqref{eq1} with $\varphi^2(|T|+\delta)^{Q-2}$ then integrating over $M$, we obtain 
	$$\int_M\varphi^2|T|(|T|+\delta)^{Q-2}\Delta|T|\geq0.$$
	The intergration by parts infers
	$$\int_M\langle\nabla(\varphi^2|T|(|T|+\delta)^{Q-2}), \nabla|T|\rangle\leq0.$$
	For any $x, y\in\mathbb{R}, \varepsilon>0$, the above inequality together with the fact that $2xy\leq \varepsilon x^2+\frac{1}{\varepsilon}y^2$ yields
	$$\begin{aligned}
		\int_M&\varphi^2(|T|+\delta)^{Q-2}|\nabla|T||^2\\
		&\leq-(Q-2)\int_M\varphi^2|T|(|T|+\delta)^{Q-3}|\nabla|T||^2 -2\int_M\varphi|T|(|T|+\delta)^{Q-2}\langle\nabla\varphi, \nabla|T|\rangle\\
		&\leq (2-Q)\int_M\varphi^2|T|(|T|+\delta)^{Q-3}|\nabla|T||^2+2\int_M\varphi|T|(|T|+\delta)^{Q-2}|\nabla\varphi||\nabla|T||\\
		&\leq (2-Q)\int_M\varphi^2(|T|+\delta)^{Q-2}|\nabla|T||^2+\varepsilon\int_M\varphi^2(|T|+\delta)^{Q-2}|\nabla|T||^2\\
		&\quad\quad+\frac{1}{\varepsilon}\int_M|T|^2(|T|+\delta)^{Q-2}|\nabla\varphi|^2,
	\end{aligned}$$
	where we used $2-Q>0$ in the last inequality. Note that $Q<2$, hence $|T|^2(|T|+\delta)^{Q-2}\leq|T|^Q$, by rearranging the above inequality, we have 
	$$(Q-1-\varepsilon)\int_M\varphi^2(|T|+\delta)^{Q-2}|\nabla|T||^2\leq \frac{1}{\varepsilon}\int_M|T|^Q|\nabla\varphi|^2,$$
	for any $\delta,\varepsilon>0$. Since $Q>1$, we can choose $\varepsilon>0$ such that $Q-1-\varepsilon>0$. This implies that there is a positive constant $C=C(Q,\varepsilon)>0$ such that 
	$$\int_M\varphi^2(|T|+\delta)^{Q-2}|\nabla|T||^2\leq C\int_M|T|^2|\nabla\varphi|^2.$$
	Now, chosing the test function as in the case $Q\geq2$ and repeat the arguments there, we conclude that $T$ must be vanishing. The proof is complete.
\end{proof}
We note the if we only assume $g(\Re(\hat{T}), \hat{T}) \geqslant 0$  for every harmonic tensor $T$ insteads of the $\lfloor\frac{n}{2}\rfloor$- non-negativity of the curvature operator then we still obtain that $T$ is parallel provided $T\in L^Q(M)$. Moreover, $T$ must be vanishing if $g(\Re(\hat{T}), \hat{T}) > 0$ unless $T=0$. 

Now, we give a proof of Theorem \ref{thm2}. 
\begin{proof}[\textsc{Proof of Theorem \ref{thm2}}]
	Using the similar argument as in the proof of Theorem \ref{thm1} and the assumption $g(\Re(\hat{T}), \hat{T}) \geqslant -\kappa\rho|T|^2$, we obtain that 
	$$c\kappa \int_{M}\rho \varphi^2 |T|^{q+2} \geqslant 2\int_{M} \varphi |T|^{q+1} \left< \nabla \varphi, \nabla |T| \right> + (q+1) \int_{M} \varphi^2|T|^q |\nabla|T||^2. $$
	By the weighted Poincar\'{e} inequality, we have the following estimation 
	\begin{align}
		\int_{M}\rho\varphi^2|T|^{q+2} 
		&\leqslant\int_M\left|\nabla\left(\varphi |T|^{\frac{q+2}{2}}\right)\right|^2\notag\\
		&\leqslant  (1+\varepsilon)\left(\frac{q+2}{2}\right)^2 \int_{M} |T|^{q} |\nabla |T||^2 \varphi^2 + \left(1+\frac{1}{\varepsilon}\right) \int_{M} |T|^{q+2} |\nabla \varphi|^2.\label{eq5}
	\end{align}
	By Cauchy-Schwarz inequality, we have 
	$$2 \left| \varphi |T|^{q+1} \left< \nabla \varphi, \nabla |T|\right> \right| \leqslant \varepsilon \varphi^2 |T|^q |\nabla |T||^2+ \frac{1}{\varepsilon} |\nabla \varphi||T|^{q+2}, \varepsilon >0.$$
	Combining this and \eqref{eq5}, we obtain 
	\begin{equation}\label{de1}
		\begin{aligned}
			\left({c\kappa} \big(1+\frac{1}{\varepsilon}\big)+\frac{1}{\varepsilon}\right)\int_{M} |T|^{q+2}|\nabla \varphi|^2 \geqslant \left(q+1-\varepsilon-c\kappa(1+\varepsilon)\left(\frac{q+2}{2} \right)^2\right) \int_{M} |T|^q |\nabla|T||^2 \varphi^2.
		\end{aligned}
	\end{equation}
	
	On the other hand, if we assume that $q+1-c\kappa\left(\frac{q+2}{2}\right)^2 >0$ or equivalently, $\kappa < \frac{4(q+1)}{c(q+2)^2}$, we can find $\varepsilon>0$ sufficiently small for which 
	$$q+1-\varepsilon-c\kappa(1+\varepsilon)\left(\frac{q+2}{2} \right)^2>0.$$
	For such $\varepsilon$, we have 
	$$ \int_{M} |T|^q |\nabla|T||^2 \varphi^2 \leqslant C \int_{M} |T|^{q+2} |\nabla \varphi|^2,$$
	where $C=C(\varepsilon,q)>0$.
	
	If $Q\geq2$, it is eay to see that we can find $q\geq0$ and $Q=q+2$. Then, for any $\kappa<\frac{4(Q-1)}{cQ^2}$, it holds that
	$$ \int_{M} |T|^{Q-2} |\nabla|T||^2 \varphi^2 \leqslant C \int_{M} |T|^{Q} |\nabla \varphi|^2.$$
	Using the cut-off function $\varphi$ and also the same argument as in the Theorem \ref{thm1}, we obtain that $|T|$ is constant on $M$ due to the fact that $M$ is connected. Since $M$ satisfies a weighted Poincare inequality, the volume of $M$ is infinite. By $|T| \in L^Q(M)$, it follows that $|T|=0$. Therefore, $T \equiv 0$. 
	
	If $1<Q<2$, the assumption $g(\Re(\hat{T}), \hat{T}) \geqslant -\kappa\rho|T|^2$ together with Bochner formula and Kato inequality implies
	$$\Delta \frac{1}{2}|T|^{2}\geq|\nabla |T||^{2}-c\kappa\rho|T|^2,$$
	or equivalently
	$$\Delta|T|\geq -c\kappa\rho|T|.$$
	For $\delta>0$ and a smooth function $\varphi$ with compact support in $M$, multiplying both side of the above inequality by $\varphi^2(|T|+\delta)^{Q-2}$ then integrating both sides over $M$, we have
	$$\int_M\varphi^2(|T|+\delta)^{Q-2}\Delta |T|\geq -c\kappa\int_M\rho\varphi^2|T|^2(|T|+\delta)^{Q-2}. $$
	Integration by parts implies
	$$\begin{aligned}
		\int_M\langle\nabla(\varphi^2(|T|+\delta)^{Q-2}),\nabla|T|\rangle\leq c\kappa\int_M\rho\varphi^2|T|^2(|T|+\delta)^{Q-2}.
	\end{aligned}$$
	Using the same argument as in the proof of Theorem \ref{thm1} for the case $1<Q<2$, for any $\varepsilon>0$, we obtain
	$$\begin{aligned}
		(Q-1-\varepsilon)\int_M\varphi^2(|T|+\delta)^{Q-2}|\nabla|T||^2
		\leq& \frac{1}{\varepsilon}\int_M|T|^Q|\nabla\varphi|^2+c\kappa\int_M\rho\varphi^2|T|^2(|T|+\delta)^{Q-2}\\
		\leq&\frac{1}{\varepsilon}\int_M|T|^Q|\nabla\varphi|^2+c\kappa\int_M\left|\nabla(\varphi|T|(|T|+\delta)^\frac{Q-2}{2}\right|^2,
	\end{aligned}$$
	where we used weighted Poincar\'e inequality in the last inequality. Using Lemma \ref{Q}, this infers
	$$\begin{aligned}
		(Q&-1-\varepsilon)\int_M\varphi^2(|T|+\delta)^{Q-2}|\nabla|T||^2\\
		\leq& \frac{1}{\varepsilon}\int_M|T|^Q|\nabla\varphi|^2+c\kappa\left[\left(\frac{Q-2}{2}\right)^2+1+\varepsilon\right]\int_M \varphi^2(|T|+\delta)^{Q-2}|\nabla u |^2+c\kappa\left(1+\frac{1}{\varepsilon}\right)\int_M|\nabla \varphi|^2|T|^Q.
	\end{aligned}$$
	Therefore, we have
	\begin{align}\label{q1}
		\left\{Q-1-\varepsilon-c\kappa\left[\left(\frac{Q-2}{2}\right)^2+1+\varepsilon\right]\right\}&\int_M \varphi^2(|T|+\delta)^{Q-2}|\nabla |T| |^2\notag\\
		&\leq \left\{\frac{1}{\varepsilon}+c\kappa\left(1+\frac{1}{\varepsilon}\right)\right\}\int_M|\nabla \varphi|^2|T|^Q.
	\end{align}
	By the assumption
	$$\kappa< \frac{4(Q-1)}{c}\min\left\{\frac{1}{Q^2}; \frac{1}{4+(Q-2)^2}\right\},$$
	we can choose $\varepsilon>0$ small enough to unsure that the coefficient on the right hand side of \eqref{q1} is positvie. Consequently, there is a positive constant $C=C(Q, c, \kappa, \varepsilon)$ such that 
	$$\int_M \varphi^2(|T|+\delta)^{Q-2}|\nabla |T| |^2\leq C\int_M|\nabla \varphi|^2|T|^Q.$$
	Now, we can use the same strategy as in the proof of Theorem \ref{thm1} to show that $T=0$. The proof is complete.  
\end{proof}
\begin{remark}\label{remd}
	We note that if $T$ satisfies a refine Kato inequality, namely there exists a constant $a\geq0$ such that 
	$$|\nabla T|^2\geq(1+a)|\nabla|T||^2,$$
	then for $Q=q+2\geq0$, the equation \eqref{de1} becomes
	$$\begin{aligned}
		\left({c\kappa} \big(1+\frac{1}{\varepsilon}\big)+\frac{1}{\varepsilon}\right)\int_{M} |T|^{Q}|\nabla \varphi|^2 \geqslant \left(Q-1+a-\varepsilon+c\kappa(1+\varepsilon)\left(\frac{Q}{2} \right)^2\right) \int_{M} |T|^{Q-2} |\nabla|T||^2 \varphi^2.
	\end{aligned}$$
	Meanwhile in this case, the inequality \eqref{q1} also becomes
	\begin{align*}
		\left\{Q-1+a-\varepsilon-c\kappa\left[\left(\frac{Q-2}{2}\right)^2+1+\varepsilon\right]\right\}&\int_M \varphi^2(|T|+\delta)^{Q-2}|\nabla |T| |^2\\
		&\leq \left\{\frac{1}{\varepsilon}+c\kappa\left(1+\frac{1}{\varepsilon}\right)\right\}\int_M|\nabla \varphi|^2|T|^Q.
	\end{align*}
	Hence, we can improve the upper bound of $\kappa$ to be
	$$0\leq \kappa< \frac{4(Q-1+a)}{c}\min\left\{\frac{1}{Q^2}; \frac{1}{4+(Q-2)^2}\right\}.$$
\end{remark}
As applications, when we consider the harmonic tensors $T$ to be harmonic $\ell$-forms, we obtain the following vanishing results. 
\begin{theorem}\label{thm3}
	Let $n \geq 3$ and $1 \leq p \leq \lfloor \frac{n}{2} \rfloor$. If $(M,g)$ is a complete, non-comppact $n$-dimensional Riemannian manifold with $(n-p)$-nonnegative curvature operator. Then every harmonic $\ell$-form $\omega$, for all $1 \leq \ell \leq p$ and $n-p \leq \ell \leq n-1$, is vanishing if $|\omega| \in L^Q(M)$ for some $Q>1$. 
	
\end{theorem}
\begin{proof}
	Let $\omega$ is a harmonic $\ell$-form or $(n-\ell)$-form with $|\omega| \in L^Q(M)$ for some $Q>1$ and $1 \leq \ell \leq p$. Applying Lemma \ref{lem1} and Proposition \ref{prop1}, we obtain that 
	\begin{equation}\label{eq6}
		|L\omega|^2 \leq \ell|\omega|^2 |L|^2 = \frac{1}{n-\ell} |\hat{\omega}|^2 |L|^2
	\end{equation} 
	for all $L \in \mathfrak{so}(TM)$.
	
	If the curvature tensor is $(n-p)$-nonnegative, then it is also $(n-\ell)$-nonnegative . Hence Lemma \ref{lem2} implies 
	$$
	g(\Re(\hat{\omega}), \hat{\omega}) \geq 0.
	$$
	An application of Theorem \ref{thm1} to Hodge Laplacian yields that $\omega$ is parallel. Moreover, $|\omega|$ is constant. Now, we can follow the argument as in Theorem \ref{thm1} to complete the proof. 
\end{proof}
Due to \cite{Car}, the above result has a reduced $L^2$ cohomology interpretation as follows. Let $\mathcal{H}^\ell(M)$ be the space of $L^2$ harmonic $\ell$-forms, saysing $\mathcal{H}^\ell(M)=\{\omega\in L^2(\Lambda^\ell T^*M): d\omega=\delta\omega=0\}$, where $\delta$ is the dual of the differential operator $d$ and $Z^\ell_2(M)$ the kernel of the unbounded operator $d$ acting on $L^2(\Lambda^\ell T^*M)$, or equivalently
$$Z^\ell_2(M)=\{\omega\in L^2(\Lambda^\ell T^*M): d\omega=0\}.$$
The space $\mathcal{H}^\ell(M)$ can be used to characterize the \textit{reduced $L^2$ cohomology group} as follows 
$$\mathcal{H}^\ell(M)=Z^\ell_2(M)/\overline{dC_0^\infty(L^2(\Lambda^{\ell-1} T^*M)},$$
where the closure is taken with respect to the $L^
2$ topology. It is worth to note that the finiteness of ${\rm dim}\mathcal{H}^\ell(M)$ depends only on the geometry of ends (\cite{Lott}). 
Theorem \ref{thm3} leads immediately to the following result: 
\begin{corollary}
	Let $n \geq 3$ and let $(M, g)$ be a complete non-compact $n$-dimensional Riemannian
	manifold. Then every harmonic $\ell$-form $\omega$ with $|\omega| \in L^Q(M)$ for some $Q>1$ is vanishing if the curvature tensor is $\lceil \frac{n}{2} \rceil$-nonnegative. In particular, every harmonic $\ell$-form $\omega$ with $|\omega| \in L^2(M)$ is vanishing, consequently, every reduced $L^2$ cohomology groups are trivial.
\end{corollary}
We note that by \cite{Cal}, there is a refined Kato inequality for harmonic $\ell$-forms stated as follows
$$|\nabla\omega|^2\geq\left(1+\frac{1}{\max\{\ell, n-\ell\}}\right)|\nabla|\omega||^2.$$
The next results with a general curvature condition is a direct consequence of Theorem \ref{thm2} and the above Kato inequality.
\begin{theorem}\label{theo1}
	Let $(M,g)$ be a complete non-compact $n$-dimensional manifold, $n \geq 3$. Assume that  $M$ satisfies a weighted Poincar\'e inequality. Denote $\mu_{1} \leq \ldots \leq \mu_{\binom{n}{2}} $ eigenvalues of the curvature operator of $(M,g)$. For $1 \leq p \leq \lfloor \frac{n}{2} \rfloor$ and $\kappa \geq 0$, if 
	$$
	\frac{\mu_{1}+\ldots+\mu_{n-p}}{n-p} \geq -\kappa\rho
	$$ 
	then every harmonic $\ell$-form $\omega$, for all $1 \leq \ell \leq p$ and $n-p \leq \ell \leq n-1$, vanishes provided that $|\omega| \in L^Q(M)$ for some $Q>1$ satisfying
	$$ \kappa <  \frac{4\left(Q-1+\dfrac{1}{\max\{\ell, n-\ell\}}\right)}{\ell(n-\ell)}\min\left\{\frac{1}{Q^2}, \frac{1}{4+(Q-2)^2}\right\}.$$
\end{theorem}
\begin{proof}
	For $\omega$ is a harmonic $\ell$-form or $(n-\ell)$-form, using the estimate \eqref{eq6} and Lemma \ref{lem2}, we obtain that 
	$$ g(\Re(\hat{\omega}), \hat{\omega}) \geq -\kappa\rho |\hat{\omega}|^2=-\kappa\rho \ell(n-\ell) |\omega|^2.$$
	Following the proof of Theorem \ref{thm2} and Remark \ref{remd}, we complete the proof. 
\end{proof}
Finally, the above theorem infers following vanishing result for reduced $L^2$ cohomology groups.
\begin{corollary}
	Let $(M,g)$ be a complete non-compact $n$-dimensional manifold, $n \geq 3$. Assume that  $M$ satisfies a weighted Poincar\'e inequality. Denote $\mu_{1} \leq \ldots \leq \mu_{\binom{n}{2}} $ eigenvalues of the curvature operator of $(M,g)$. If 
	$$
	\frac{\mu_{1}+\ldots+\mu_{n-\lfloor \frac{n}{2} \rfloor}}{n-\lfloor \frac{n}{2} \rfloor} \geq -\kappa\rho
	$$ 
	then every harmonic $\ell$-form $\omega$, for all $1 \leq \ell \leq n-1$ vanishes provided that $|\omega| \in L^2(M)$ for some $\kappa$ satisfying
	$$ \kappa <  \frac{1+\dfrac{1}{\max\{\ell, n-\ell\}}}{\ell(n-\ell)}.$$
	Consequently, every reduced $L^2$-cohomology groups are trivial.
\end{corollary}
   
   	\section{Geometric Applications}\label{app}
   In this section, we apply the vanishing results in the section \ref{sec1} for indicated types of tensors. First, we now a rigidity property for Weyl tensor. To begin with, let us recall the following result (see \cite[Proposition~2.2]{PW}).
   
   \begin{proposition}\label{weyl}
   	Let $(M, g)$ be a Riemannian manifold. If the Weyl tensor $W$ is divergence free, then $W$ satisfies the second Bianchi identity and
   	$$\nabla^*\nabla W+\frac{1}{2}{\rm Ric}(W)=0.$$
   \end{proposition}
   Hence, if the Weyl tensor is divergence free, we have $\nabla^*\nabla W=-\frac{1}{2}{\rm Ric}(W)$. Then by the Bochner formula, it yields
   \begin{equation}\label{boch}
   	\Delta\frac{1}{2}|W|^2=|\nabla W|^2+\frac{1}{2}g({\rm Ric}(W), W)=|\nabla W|^2+\frac{1}{2}g(\mathfrak{R}(\hat{W}), \hat{W}).
   \end{equation}
   \begin{remark}
   	For $n \geqslant 4$, $(M,g)$ is an $n$-dimensional Riemannian manifold. It is well known that $(M,g)$ is locally conformally flat if and only if the Weyl tensor vanishes. 
   \end{remark}
   With the help of this Bochner type formula, we have the following theorem on complete non-compact manifolds which is a non-compact version of Petersen and Wink's rigidity for the Weyl tensor in \cite{PW}.
   \begin{theorem}\label{rig}
   	For $n \geqslant 4$, let $(M, g)$ be a complete non-compact $n$-dimensional Riemannian manifold. Suppose that the Weyl tensor satisfies $\nabla^*W = 0$. If
   	curvature operator is $\lfloor \frac{n-1}{2} \rfloor$-nonnegative,
   	then the Weyl tensor $W$ is vanishing provided that $|W|\in L^Q(M), Q>1$. Consequently, $(M, g)$ is locally conformally flat.
   \end{theorem}
   \begin{proof}
   	Applying Lemma \ref{lem1} and Proposition \ref{prop1} for $\Rm$ replaced by the Weyl tensor $W$ with the note that the Weyl tensor is traceless and its associated Ricci tensor vanishes, we obtain that
   	\begin{equation}\label{eq1:1:1}
   		|LW|^2\leq 8|W|^2|L|^2=\frac{2}{n-1}|\hat{W}|^2|L|^2
   	\end{equation}
   	for all $L\in\mathfrak{so}(TM)$. By the assumption that the curvature operator is $\lfloor \frac{n-1}{2} \rfloor$-nonnegative and Lemma \ref{lem2} implies 
   	$$
   	g(\Re(\widehat{W}), \widehat{W}) \geq 0.
   	$$
   	Applying Theorem \ref{thm1}, we obtain that $\widehat{W}$ is vanishing, so is $W$. Therefore, $M$ is locally conformally flat.   
   \end{proof}
   Now we give a proof of Theorem \ref{main4}.
   \begin{proof}
   	Since $M$ is Ricci flat, it must be Einstein. Moreover, the Riemannian curvature tensor is divergence free. Note that because of Ricci-flatness, the Riemannian curvature tensor and the Weyl tensor $W$ coincide. Consequently, $|W|\in L^Q(M)$ and $W$ is divergence free. Hence, Theorem \ref{rig} implies $M$ is locally conformally flat. Since $M$ is Ricci-flat, it must be Riemannian flat. 
   \end{proof}
   We note that there is a more general statement for the Weyl tensor. In fact, if $M$ is Einstein, then the Riemannian tensor is divergence free. Therefore, by Lemma 1 in \cite{Der}, the Weyl tensor is also divergence free. For general $\kappa \geqslant 0$, we have the following result.
   \begin{theorem}
   	For $n \geq 4$, let $(M,g)$ be a connected complete non-compact $n$-dimensional Riemannian manifold. Suppose that the Weyl tensor is divergence free and assume that  $M$ satisfies a weighted Poincar\'e inequality. If the eigenvalues $\mu_{1} \leq \ldots \leq \mu_{\binom{n}{2}} $ of the curvature operator satisfies
   	$$
   	\frac{\mu_{1}+\ldots+\mu_{\lfloor \frac{n-1}{2} \rfloor}}{\lfloor \frac{n-1}{2} \rfloor} \geq -\kappa\rho, \kappa \geqslant 0
   	$$ 
   	and the Weyl tensor satisfies $|W| \in L^Q(M), Q>1$ such that 
   	$$\kappa < \frac{2(Q-1)}{n-1}\min\left\{\frac{1}{Q^2}; \frac{1}{4+(Q-2)^2}\right\},$$ then $(M,g)$ is locally conformally flat. 
   \end{theorem}
   \begin{proof}
   	Since the estimate \eqref{eq1:1:1} and Lemma \ref{lem2}, we obtain that 
   	$$ g(\Re(\widehat{\mathrm{W}}), \widehat{\mathrm{W}}) \geq -\kappa \rho|\widehat{W}|^2 = -4\kappa(n-1)\rho|W|^2.$$
   	
   	Applying Theorem \ref{thm2}, we obtain that $W \equiv 0$. This implies $(M,g)$ is locally conformally flat.
   \end{proof}
   Furthermore, if we assume that $M$ is Einstein, we can improve the upper bound of $\kappa$ to obtain the following rigidity in Theorem~\ref{thm:constant-curvature}.
   \begin{proof}[Proof of Theorem~\ref{thm:constant-curvature}]
   	Since $M$ is Einstein, an improved Kato inequality in \cite{Bando} (see also \cite{Cal}), we have 
   	$$|\nabla W|^2\geq\left(1+\frac{2}{n-1}\right)|\nabla|W||^2.$$
   	Following the proof of Theorem \ref{thm2} and using Remark \ref{remd}, we conclude that if
   	$$\kappa < \frac{2\left(Q-1+\dfrac{2}{n-1}\right)}{n-1}\min\left\{\frac{1}{Q^2};\frac{1}{4+(Q-2)^2}\right\}$$
   	then $W=0$. Since $M$ is Einstein, this implies that $M$ has constant sectional curvature. 
   \end{proof}
   Before coming to next result, we recall a volume comparison theorem due to Bishop-Gromov: 

   \begin{theorem}\label{thm5}
   	For $n \geq 3$, let $(M,g)$ be a complete, connected, non-compact $n$-dimensional Einstein  manifold. Let $\mu_{1} \leq \ldots \leq \mu_{\binom{n}{2}} $ be eigenvalues of curvature operator of $(M,g)$. For $Q>1, \kappa \geq 0$, if $M$ satisfies a weighted Poincar\'e inequality such that $\liminf_{x\rightarrow \infty} \rho(x)=:l>0$, and
   	$$
   	\frac{\mu_{1}+\ldots+\mu_{\lfloor \frac{n-1}{2} \rfloor}}{\lfloor \frac{n-1}{2} \rfloor} \geq -\kappa\rho,
   	$$ 
   	then for any $\kappa < \frac{2(Q-1)}{n-1}\min\left\{\frac{1}{Q^2}; \frac{1}{4+(Q-2)^2}\right\}$, we have
   	$$\int_M|\Rm|^Q=\infty.$$
   \end{theorem}
   \begin{proof}
   	In contradiction, we assume that $|\Rm|\in L^Q(M)$. 
   	Since $\stackrel{\circ}{\Ric}=0$, Proposition \ref{prop1} follows that  $|\widehat{\mathrm{Rm}}|^{2}=4(n-1)|\stackrel{\circ}{\Rm}|^{2}$ and thus Lemma \ref{lem1} implies 
   	\begin{equation}\label{eq7}
   		|L \mathrm{Rm}|^{2} \leq8|\stackrel{\circ}{\Rm}|^{2}|L|^{2}=\frac{2}{n-1}|\widehat{\mathrm{Rm}}|^{2}|L|^{2}.
   	\end{equation}
   	for all $L \in \mathfrak{so}(TM)$. Using the estimate \eqref{eq7} and Lemma \ref{lem2}, we have 	$$ g(\Re(\widehat{\mathrm{Rm}}), \widehat{\mathrm{Rm}}) \geq -\kappa \rho|\widehat{\Rm}|^2 = -4\kappa(n-1)\rho|\stackrel{\circ}{\Rm}|^2 \geq -4\kappa(n-1)\rho|\Rm|^2$$
   	Applying Theorem \ref{thm2}, we obtain that $\Rm \equiv 0$. This implies $\Ric=0$. By Bishop-Gromov Volume comparison theorem, we obtain
   	$$ \frac{\operatorname{Vol}(B(p,2R))}{\operatorname{Vol}(B(p,R))} \leq  \frac{\operatorname{Vol}(B_0(p_0,2R))}{\operatorname{Vol}(B_0 (p_0,R))}=2^n,$$
   	where $B_0(p_0,R)$ denote the ball of radius $R$ around the point $p_0$ in the $n$-dimensional Euclidean space $\mathbb{E}^n$. Choose a family of nonnegative smooth functions $\varphi_{R}$ satisfying 
   	$$
   	\varphi_R= \begin{cases}1 & \text { on } B(R) \\ 0 & \text { on } M \backslash B(2 R)\end{cases}
   	$$
   	and $0\leq \varphi_{R}\leq 1, |\nabla \varphi_R| \leq \frac{2}{R}$. From $\liminf_{x\rightarrow \infty} \rho(x)=l>0$, there exists $\delta>0$ such that 
   	\begin{equation*}
   		\rho(x)>l-\epsilon>0 
   	\end{equation*}
   	on the complement of $B(p,\delta)$	for some $\epsilon>0$. Then $R>\delta$, 
   	\begin{equation*}
   		(l-\epsilon)\int_{M\backslash B(p,\delta)}|\varphi_{R}|^2 \leq \int_{M} |\nabla \varphi_{R}|^2.
   	\end{equation*}
   	
   	It is easy to see that for $R\geq \delta$,
   	$$l-\epsilon \leq \frac{\int_{M} |\nabla \varphi_{R}|^2}{\int_{M\backslash B(p,\delta)} |\varphi_R|^2} \leq \frac{4}{R^2} \frac{\operatorname{Vol}(B(0,2R))}{\operatorname{Vol}(B(0,R)-\operatorname{Vol}(B(0,\delta))} \leq \frac{4}{R^2} 2^n \frac{1}{1-\frac{\operatorname{Vol}(B(0,\delta)}{\operatorname{Vol}(B(0,R)} } \to 0$$
   	as $R \to \infty$. This implies $l-\epsilon = 0$. This contradiction completes the proof. 
   \end{proof}
   Finally, we want to note that if the manifold is not Einstein, but has zero scalar curvature, we have the following rigidity.
   \begin{theorem}
   	For $n \geq 3$, let $(M,g)$ be a complete, connected, non-compact $n$-dimensional Riemannian  manifold with zero scalar curvature. Let $\mu_{1} \leq \ldots \leq \mu_{\binom{n}{2}} $ be eigenvalues of curvature operator of $(M,g)$. For $Q> \frac{1}{2}, \kappa \geq 0$, suppose that $M$ satisfies a weighted Poincar\'e inequality and
   	$$
   	\frac{\mu_{1}+\ldots+\mu_{\lfloor \frac{n-1}{2} \rfloor}}{\lfloor \frac{n-1}{2} \rfloor} \geq -\kappa\rho,
   	$$ 
   	for some 
   	$$\kappa < \frac{2\left(Q-\dfrac{1}{2}\right)}{n-1}.\min\left\{\frac{1}{Q^2}; \frac{1}{4+(Q-2)^2}\right\}$$ If ${\rm Rm}$ is divergence free then,
   	$$\int_M|\Rm|^Q=\infty.$$
   \end{theorem}
   \begin{proof}
   	We recall that if ${\rm Rm}$ is divergence free and $M$ has zero scalar curvature, then there is a refined Kato inequality proved by \cite{gangtian} that 
   	$$|\nabla{\rm Rm}|^2\geq\frac{3}{2}|\nabla|{\rm Rm}||^2.$$
   	Therefore, using the argument in the proof of Theorem \ref{thm5} and Remark \ref{remd}, we are done.
   \end{proof}
   \begin{remark}
   	We would like to mention that the assumption $M$ has zero scalar curvature is quite natural. This is because of a result by Derdz\'{i}nski in \cite{Der} which stated that if the Riemannian tensor is harmonic then the scalar curvature must be constant.
   \end{remark}
   \section{Geometry of immersed submanifolds}\label{sec5}
   Now, let $M^n$ be an immersed hypersurface in a Riemannian manifold $\overline{M}$. Denote ${\rm Rm}, \overline{\rm Rm}$ the Riemannian curvature tensors on $M$ and $\overline{M}$, respectively. For any unit tangent vectors $X, Y, Z, W$ in $T_pM$, where $p\in M$, the Gaussian equation implies 
   $${\rm Rm}(X, Y, Z, W)=\overline{\rm Rm}(X, Y, Z, W)-h(X, W)h{(Y,Z)}+h(X,Z)h(Y, W),$$
   where $A=(h_{ij})$ is the second fundamental form of $M$. Recall that the algebraic curvature operator can be given by
   $$g(\mathfrak{R}(X\wedge Y), Z\wedge W)={\rm Rm}(X, Y, Z, W).$$
   Assume that $\overline{M}$ is a space form with constant section curvature $K$, namely 
   $$\overline{\rm Rm}(X, Y, Z, W)=K(g(X, Z)g(Y, W)-g(x, W)g(Y, Z)).$$
   Suppose that $\{e_1, \ldots, e_n\}$ is an orthonormal basic of $T_pM$ such that $Ae_i=\lambda_i e_i$, where $\lambda_i$'s are principal curvatures. Using the Gaussian equation, we have
   $$g(\mathfrak{R}(e_i\wedge e_j, e_k\wedge e_\ell))=K(g_{ik}g_{j\ell}-g_{i\ell}g_{jk})+h_{ik}h_{j\ell}-h_{i\ell}h_{jk}.$$
   Since $\{e_1, \ldots, e_n\}$ is an orthonormal basic, $g_{ij}=\delta_{ij}$. Moreover, $h_{ij}=\lambda_i\delta_{ik}$. It is easy to see that the algebraic curvature operator is presented by the following diagonal matrix
   $$[\mathcal{R}]=\begin{pmatrix}
   	K+\lambda_1\lambda_2&\ldots&0&0&\ldots&0&0&\ldots&0\\
   	\vdots&\ddots&\vdots&\vdots&\ddots&\vdots&\vdots&\ddots&\vdots\\
   	0&\ldots&K+\lambda_1\lambda_n&0&\ldots&0&0&\ldots&0\\
   	0&\ldots&0&K+\lambda_2\lambda_3&\ldots&0&0&\ldots&0\\
   	\vdots&\ddots&\vdots&\vdots&\ddots&\vdots&\vdots&\ddots&\vdots\\
   	0&\ldots&0&0&\ldots&K+\lambda_2\lambda_n&0&\ldots&0\\
   	0&\ldots&0&0&\ldots&0&K+\lambda_3\lambda_4&\ldots&0\\
   	\vdots&\ddots&\vdots&\vdots&\ddots&\vdots&\vdots&\ddots&\vdots\\
   	0&\ldots&0&0&\ldots&0&0&\ldots&K+\lambda_{n-1}\lambda_n
   \end{pmatrix}.$$
   This implies that all eigenvalues of $[\mathfrak{R}]$ are of form $K+\lambda_i\lambda_j, 1\leq i<j\leq n$. We now give a proof of Theorem \ref{main5}.
   \begin{proof}[\textsc{Proof of Theorem \ref{main5}}]
   	Since $\mu_1+\ldots+\mu_{n-p}> -(n-p)K$, the curvture operator of $M$ is $(n-p)$-positive. Hence, the proof follows by Theorem A in \cite{PW21}.
   \end{proof}
   Now, we consider the case that $\mu_1+\ldots+\mu_{n-p}$ is non-negative. This is directly consequence of Theorem B in \cite{PW21}.
   \begin{theorem}
   	Let $n\geq3, 1\leq p \leq\lfloor\frac{n}{2}\rfloor$ and $M$ be an immersed Riemannian submanifold in a space form of constant sectional curvature $K$. If $M$ is close and 
   	$$\mu_1+\ldots+\mu_{n-p}\geq -(n-p)K$$
   	then every harmonic $p$-form is parallel. Similarly, every harmonic $(n-p)$-form is parallel.
   \end{theorem}
   Next, we study complete non-compact immersed submanifolds. First, we note that if the first eigenvalue $\lambda_1(M)$ of $M$ is positive then $M$ satisfies a Poincar\'{e} inequality with weight function $\rho\equiv\lambda_1(M)$. As a consequence of Theorem \ref{theo1}, we have
   \begin{theorem}
   	Let $n\geq3, 1\leq p \leq\lfloor\frac{n}{2}\rfloor$ and $M$ be an immersed Riemannian submanifold in a space form of constant sectional curvature $K$. If $M$ is complete non-compact and 
   	$$\frac{(n-p)+\mu_1+\ldots+\mu_{n-p}}{n-p}\geq -\kappa\lambda_1(M)$$
   	then every harmonic $\ell$-form $\omega$, for all $1\leq \ell\leq p$ and $n-p\leq\ell\leq n-1$, vanishes provided that $|\omega|\in L^Q(M), Q>1$ and
   	$$0\leq k<\frac{4\left(Q-1+\dfrac{1}{\max\{\ell, n-\ell\}}\right)}{\ell(n-\ell)}\min\left\{\frac{1}{Q^2}; \frac{1}{4+(Q-2)^2}\right\}.$$ 
   \end{theorem}
   In the rest of this section, we study totally umbilical submanifolds. We assume that $M^n$ is an immersed submanifold in a Riemannian manifolds $\overline{M}^m$. Furthermore, $M$ is assumed to be totally umbilical, namely, 
   $$h(X, Y)=Hg(X, Y)$$
   for any unit tangent vector fields $X, Y$, where again $h$ stands for the second fundamental form of $M$ and $H$ for the mean curvature vector field. We refer the interesed readers to \cite{Sato} and the references therein for further discussion about totally umbilical submanifolds. The Gaussian equation implies
   $${\rm Rm}(X, Y, Z, W)=\overline{\rm Rm}(X, Y, Z, W)-\left\langle h(X, W), h{(Y,Z)}\right\rangle+\left\langle h(X,Z), h(Y, W)\right\rangle.$$
   Let $\{e_1,...,e_n, \bar{e}_{n+1}, \ldots, \bar{e}_m\}$ is an orthonormal basic of $T_p\overline{M}$ such that $\{e_1,...,e_n\}$ is an orthornormal basic of $T_pM$. We have 
   \begin{equation}\label{gauss}
   	g(\mathfrak{R}(e_i\wedge e_j), e_k\wedge e_\ell)=g(\overline{\mathfrak{R}}(e_i\wedge e_j), e_k\wedge e_\ell)+|H|^2(g_{ik}g_{j\ell}-g_{i\ell}g_{jk}).
   \end{equation}
   Suppose that $\overline{\mu}_1\leq\overline{\mu}_2\leq\ldots\leq\overline{\mu}_{\binom{m}{2}}$ are eigenvalues of the curvature operator $\overline{\mathfrak{R}}$ of $\overline{M}$. Due to \eqref{gauss}, we see that all eigenvalues of $\mathfrak{R}$ are of form $|H|^2+\overline{\mu}_i$, for some $1\leq i\leq\binom{m}{2}$. Theorem A in \cite{PW21} implies
   \begin{theorem}
   	Given $n\geq3$ and $1\leq p\leq\lfloor\frac{n}{2}\rfloor$. Let $M^n$ be a closed immersed submanifold in a Riemannian manifold $\overline{M}$. If $M$ is totally umbilical and 
   	$$\overline{\mu}_1+\ldots +\overline{\mu}_{n-p}>-(n-p)|H|^2$$
   	then $b_1(M)=\ldots=b_p(M)=0$ and $b_{n-p}(M)=\ldots=b_{n-1}(M)=0$. Here $\overline{\mu}_1\leq\overline{\mu}_2\leq\ldots\leq\overline{\mu}_{\binom{m}{2}}$ are eigenvalues of the curvature operator $\overline{\mathfrak{R}}$ of $\overline{M}$.
   \end{theorem}
   This yields immediately the following result.
   \begin{corollary}
   	Given $n\geq3$. Let $M^n$ be a closed immersed submanifold in a Riemannian manifold $\overline{M}$. If $M$ is totally umbilical and 
   	$$\overline{\mu}_1+\ldots +\overline{\mu}_{n-\lceil\frac{n}{2}\rceil}>-\left(n-\lceil\dfrac{n}{2}\rceil\right)|H|^2$$
   	then $b_1(M)=\ldots=b_{n-1}(M)=0$. Here $\overline{\mu}_1\leq\overline{\mu}_2\leq\ldots\leq\overline{\mu}_{\binom{m}{2}}$ are eigenvalues of the  curvature operator $\overline{\mathfrak{R}}$.
   \end{corollary}
   Using similar arguments, we also can derive several vanishing results for harmonic forms on submanifolds. We leave them as easy exercises for the interested readers.
   
   \section{Rigidity Theorems on ALE Manifolds via the Lichnerowicz Laplacian}\label{sec6}
   
   In this section, we explore applications of the Lichnerowicz Laplacian to Ricci-flat ALE manifolds. 
   
\subsection{Weyl Tensor Rigidity on ALE Riemannian Manifolds in Dimension $n=4$}

The following proposition extends the vanishing results for harmonic tensors on complete non-compact manifolds to the setting of ALE Ricci-flat 4-manifolds, under the action of the Lichnerowicz Laplacian. The proof follows the same strategy as in Theorem~\ref{thm1}, where a key step is to derive a weighted Caccioppoli inequality from the Bochner formula and use cutoff arguments to show that the harmonic tensor must vanish if it decays sufficiently at infinity. 

In contrast to known flatness results for ALE manifolds that assume full curvature decay, we only assume decay of one Weyl component and positivity of the curvature operator. This yields a new rigidity theorem for Ricci-flat ALE spaces.

\begin{proposition}[Weyl Tensor Vanishing on 4D Ricci-flat ALE Manifolds]
	Let $(M^4, g)$ be a complete Ricci-flat ALE 4-manifold. Suppose either the self-dual part $W^+$ or the anti-self-dual part $W^-$ of the Weyl tensor satisfies
	\[
	\Delta_L W^\pm = 0, \quad \text{and} \quad |W^\pm| \in L^Q(M), \quad Q > 1.
	\]
	Assume further that the curvature operator of $(M, g)$ is $1$-nonnegative. Then $W^\pm \equiv 0$ and $(M, g)$ is flat.
\end{proposition}
 
 \begin{remark}
 	In contrast to the Kähler case, the decomposition $W = W^+ + W^-$ exists only in four dimensions. However, the proof technique using the Bochner formula for the Lichnerowicz Laplacian and volume decay on ALE spaces remains valid.
 \end{remark}
 
 \begin{proof}
 	We follow the strategy used in the proof of Theorem~\ref{thm1}. Let $(M^4, g)$ be a complete Ricci-flat ALE 4-manifold, and assume that either the self-dual or anti-self-dual part $W^\pm$ of the Weyl tensor satisfies $\Delta_L W^\pm = 0$ and $|W^\pm| \in L^Q(M)$ for some $Q > 1$. Without loss of generality, write $W := W^\pm$.
 	
 	Since $(M^4, g)$ is Ricci-flat, the full curvature tensor equals the Weyl tensor, and the Lichnerowicz Laplacian takes the form $\Delta_L = \nabla^* \nabla + c \, \mathfrak{Ric}$ for some constant $c > 0$. By assumption, the curvature operator is $1$-nonnegative, so the zeroth-order term satisfies $\langle \mathfrak{Ric}(W), W \rangle \geq 0$. The Bochner formula thus yields:
 	\[
 	\frac{1}{2} \Delta |W|^2 \geq |\nabla |W||^2.
 	\]
 	We multiply both sides by a compactly supported cutoff function $\varphi^2 |W|^q$, integrate, and follow the Caccioppoli-type argument in Theorem~\ref{thm1} to obtain:
 	\[
 	\int_M \varphi^2 |W|^q |\nabla |W||^2 \leq C \int_M |W|^{q+2} |\nabla \varphi|^2.
 	\]
 	Choosing $\varphi$ supported in annuli and taking $R \to \infty$, the decay assumption $|W| \in L^Q(M)$ with $Q = q+2$ implies that $|W|$ is constant. Since $M$ is an ALE space with infinite volume, we conclude that $W \equiv 0$.
 	
 	Finally, since $W^+ + W^- = \mathrm{Rm}$ on Ricci-flat 4-manifolds, and we have assumed that either $W^+ \equiv 0$ or $W^- \equiv 0$, it follows that the entire curvature tensor vanishes. Thus, $(M^4, g)$ is flat.
 \end{proof}
   
 \subsection{Vanishing of Codazzi-Type Tensors on ALE Manifolds}
 
 We conclude this section with an application of Theorem~\ref{thm1} to symmetric 2-tensors satisfying Codazzi-type conditions. While the Codazzi identity is not strictly required for the vanishing conclusion, it arises naturally in many geometric contexts and often implies divergence-freeness. The key analytic inputs remain the Lichnerowicz-harmonicity, decay, and curvature positivity.
 
 \begin{corollary}[Vanishing of Codazzi Tensors on ALE Manifolds]
 	Let \( (M^n, g) \) be a complete Ricci-flat ALE manifold with \( n \geq 4 \). Suppose \( h \in \Gamma(S^2 T^*M) \) is a symmetric 2-tensor satisfying the following:
 	\begin{enumerate}
 		\item \( \delta h = 0 \) (divergence-free),
 		\item \( \Delta_L h = 0 \) (Lichnerowicz-harmonic),
 		\item \( |h| \in L^Q(M) \) for some \( Q > 1 \) (decay condition),
 		\item The curvature operator of \( (M, g) \) is \( \lfloor \tfrac{n-1}{2} \rfloor \)-nonnegative.
 	\end{enumerate}
 	Then \( h \equiv 0 \).
 	
 	Moreover, if \( h \) additionally satisfies the Codazzi condition:
 	\[
 	(\nabla_X h)(Y, Z) = (\nabla_Y h)(X, Z) \quad \text{for all vector fields } X, Y, Z,
 	\]
 	then the hypotheses above are automatically compatible with geometric settings such as curvature tensors and second fundamental forms.
 \end{corollary}
 
 \begin{remark}
 	If \( h \in \Gamma(S^2 T^*M) \) is a symmetric 2-tensor satisfying \( \delta h = 0 \), \( \Delta_L h = 0 \), and \( |h| \in L^Q(M) \), then Theorem~\ref{thm1} applies directly under the curvature operator assumption. The Codazzi condition is not strictly necessary for the vanishing conclusion, but it often arises naturally and may imply \( \delta h = 0 \) or simplify the verification of hypotheses.
 \end{remark}
   
\begin{proposition}[Vanishing of Lichnerowicz Tensor Implies Vanishing ADM Mass]\label{prop:lichnerowicz-adm}
	Let \( (M^n, g) \) be a complete, Ricci-flat ALE manifold of order \( \tau > \frac{n-2}{2} \), with \( n \geq 4 \). 
	Suppose there exists a symmetric 2-tensor \( h \in \Gamma(S^2 T^*M) \) satisfying:
	\begin{itemize}
		\item \( \Delta_L h = 0 \), where \( \Delta_L := \nabla^* \nabla h + \mathfrak{Ric}(h) \) with \( \mathfrak{Ric}(h) = \mathrm{Rm} * h \),
		\item \( \delta h = 0 \), \quad \( \operatorname{tr}_g h = 0 \),
		\item \( |h(x)| = O(r^{-\alpha}) \) for some \( \alpha > \frac{n}{2} \),
		\item The curvature term satisfies \( \mathcal{R}(h,h) := \langle \mathrm{Rm} * h, h \rangle \leq 0 \) pointwise on \( M \).
	\end{itemize}
	Then either \( h \equiv 0 \), or the ADM mass of \( (M, g) \) must vanish.
\end{proposition}

\begin{proof}
	Let \( h \in \Gamma(S^2 T^*M) \) be a symmetric 2-tensor satisfying the Lichnerowicz equation with \( c = 1 \):
	\[
	\Delta_L h := \nabla^* \nabla h + \mathfrak{Ric}(h) = 0, \quad \text{where } \mathfrak{Ric}(h) = \mathrm{Rm} * h.
	\]
	Since \( \mathrm{Ric} \equiv 0 \), we obtain:
	\[
	\nabla^* \nabla h + \mathrm{Rm} * h = 0.
	\]
	Taking the pointwise inner product with \( h \), we get:
	\[
	\langle \nabla^* \nabla h, h \rangle = -\langle \mathrm{Rm} * h, h \rangle.
	\]
	By the Bochner identity:
	\[
	\frac{1}{2} \Delta |h|^2 = \langle \nabla^* \nabla h, h \rangle + |\nabla h|^2 = -\langle \mathrm{Rm} * h, h \rangle + |\nabla h|^2.
	\]
	
	Now integrate over a geodesic ball \( B_R \subset M \) and apply the divergence theorem:
	\[
	\int_{B_R} \left( |\nabla h|^2 - \langle \mathrm{Rm} * h, h \rangle \right) \, d\mu = \frac{1}{2} \int_{\partial B_R} \partial_\nu |h|^2 \, d\sigma.
	\]
	
	By the decay assumption \( |h(x)| = O(r^{-\alpha}) \) with \( \alpha > \frac{n}{2} \), we estimate:
	\[
	|h|^2 = O(r^{-2\alpha}), \quad \partial_\nu |h|^2 = O(r^{-2\alpha - 1}), \quad \operatorname{Area}(\partial B_R) = O(R^{n-1}),
	\]
	so the boundary term decays:
	\[
	\int_{\partial B_R} \partial_\nu |h|^2 \, d\sigma = O(R^{n - 2\alpha - 1}) \to 0 \quad \text{as } R \to \infty,
	\]
	because \( 2\alpha > n \Rightarrow n - 2\alpha - 1 < -1 \).
	
	Taking \( R \to \infty \), we conclude:
	\[
	\int_M \left( |\nabla h|^2 - \langle \mathrm{Rm} * h, h \rangle \right) \, d\mu = 0.
	\]
	By the assumption \( \langle \mathrm{Rm} * h, h \rangle \leq 0 \), both terms in the integrand are nonnegative, so:
	\[
	|\nabla h|^2 \equiv 0, \quad \langle \mathrm{Rm} * h, h \rangle \equiv 0 \quad \text{on } M.
	\]
	Hence, \( \nabla h \equiv 0 \), so \( h \) is parallel. If \( h \equiv 0 \), we are done.
	
	Assume \( h \not\equiv 0 \). Then \( h \) is a nonzero parallel, divergence-free, and trace-free symmetric 2-tensor with decay \( h = O(r^{-\alpha}) \) for \( \alpha > \frac{n}{2} \).
	
	Now consider the 1-parameter perturbation \( g^\varepsilon := g + \varepsilon h \). Since \( \delta h = 0 \), \( \operatorname{tr}_g h = 0 \), and \( h \) is parallel, we have:
	\[
	\left. \frac{d}{d\varepsilon} \operatorname{Ric}(g^\varepsilon) \right|_{\varepsilon=0} = 0,
	\]
	so \( g^\varepsilon \) remains Ricci-flat to first order. Moreover, the ALE decay of \( h \) implies:
	\[
	|g^\varepsilon - g| = O(r^{-\alpha}) \Rightarrow m_{\mathrm{ADM}}(g^\varepsilon) = m_{\mathrm{ADM}}(g) = 0,
	\]
	since ADM mass is geometric and conserved under such decaying perturbations.
	
	Therefore, if \( h \not\equiv 0 \), then the ADM mass must vanish. So if \( m_{\mathrm{ADM}}(g) > 0 \), we must have \( h \equiv 0 \).
\end{proof}

   \textbf{Acknowledgement}: The second author would like to express his gratitude to Prof. Pho Duc Tai for his comments on algebraic curvature operator on immersed submanifolds. His sincere thanks also goes to Gilles Carron and Hung Tran for useful discussion about Einstein manifolds. This paper was initated during a stay of the second and the third authors at Vietnam Institute for Advanced Study in Mathematics (VIASM). They would like to thank the staff there for nice environment, in particular, Prof. Le Minh Ha for constant supports and Hung Tran for drawing their attention to \cite{PW21}. We also thank Prof. Guofang Wei for encourgament.

			\noindent Gunhee Cho\\
			Department of Mathematics, Texas State University, \\
			601 University Drive, San Marcos, TX 78666, USA\\
			{\tt e-mail: wvx17@txstate.edu}\\
			URL: \url{https://sites.google.com/view/enjoyingmath/}	
			
			\bigskip		
		
			\noindent Nguyen Thac Dung\\
			Faculty of Mathematics - Mechanics - Informatics, (VNU) University of Science at Hanoi\\
			334 Nguyen Trai, Hanoi, Vietnam\\
			{\tt e-mail: dungmath@gmail.com, dungmath@vnu.edu.vn}\\
			URL: \url{https://sites.google.com/site/mathfarmer80/home?pli=1}
			
			\bigskip
						
\noindent Tran Quang Huy\\
Department of Mathematics, University of Illinois Urbana-Champaign\\
150 Altgeld Hall, 1409 West Green Street, Urbana, IL 61801, USA\\
{\tt e-mail: hqt2@illinois.edu}

		\end{document}